\documentclass[12pt, leqno]{amsart}
\DeclareSymbolFont{rsfs}{U}{rsfs}{m}{n}
\DeclareSymbolFontAlphabet{\rsfsmathscr}{rsfs}
\usepackage[mathscr]{eucal}
\usepackage{color}

\usepackage{amsmath,amscd,amsthm,amssymb}
\usepackage{latexsym, amssymb,  amsmath}

\def\R{{\mathbb{R}}}
\def\Rn{{\mathbb{R}^n}}

\def\i{\infty}
\def\B{{\mathcal{B}}}
\def\ds{\displaystyle}

\def\L{\mathcal{L}}
\def\SS{{\mathbb{S}}}
\def\M{\mathcal{M}}
\def\K{\mathcal{K}}

\def\E{\mathcal{E}}

\def\ba{\mathbf{ a}}
\def\loc{\rm{ loc}}

\def\sm {\setminus}
\def\wtl{\widetilde}
\def\T{\mathcal{T}}
\def\KK{\mathfrak{K}}
\def\tl {\tilde}
\def\CC{\mathfrak{C}}

\def\Lploc{L_p^{\rm loc}(\Rn)}
\def\Lploc+{L_p^{\rm loc}(\R^n_+)}
\def\Lpwloc+{L_{p,w}^{\rm loc}(\R^n_+)}

\def\supp{{\rm supp}}

\newcommand{\ess}{\mathop{\rm ess \; sup}\limits}
\newcommand{\es}{\mathop{\rm ess \; inf}\limits}

\newtheorem{theorem}{Theorem}{}
\newtheorem{corollary}{Corollary}{}
\newtheorem{definition}{Definition}{}
\newtheorem{remark}{Remark}{}
{}
\newtheorem{lemma}{Lemma}{}
{}

\numberwithin{equation}{section}

\makeatother

\begin{document}

\baselineskip=20pt

\title[The Dirichlet problem in weighted spaces]{ The Dirichlet problem in a class of generalized weighted spaces}

\author[V. Guliyev]{Vagif  S. Guliyev}
\address{Ahi Evran University, Department of Mathematics,
 Kirsehir, Turkey}
\email{vagif@guliyev.com}

\author[M. Omarova]{Mehriban Omarova}
\address{ Institute of Mathematics and Mechanics
 of NAS of Azerbaijan, Baku}
\email{ mehriban\_omarova@yahoo.com}

\author[L.G. Softova]{Lubomira G. Softova}
\address{Department of Civil Engineering, Design, Construction and Environment,
Second University  of Naples,
 Italy}
\email{luba.softova@unina2.it}



\maketitle

\begin{abstract}
We show continuity in generalized weighted Morrey spaces $M_{p,\varphi}(w)$ of sub-linear integral operators
generated by  some classical integral operators and commutators.  The obtained estimates are used to study global regularity of the solution
of the Dirichlet problem for linear  uniformly elliptic operators with discontinuous  data.
 \end{abstract}

\section{Introduction}\label{secIntr}
\setcounter{theorem}{0}

In the present work we study the global regularity in {\it generalized weighted Morrey spaces}  $M_{p,\varphi}(w)$ of the solutions of a class of elliptic partial differential equations (PDEs). Recall that the classical {\it Morrey spaces} $L_{p,\lambda}$ were introduced  by Morrey  in  \cite{Morrey} in order to study the local H\"older regularity  of the solutions of elliptic  systems.
 In \cite{ChFra} Chiarenza and Frasca show
 boundedness    in  $L_{p,\lambda}(\R^n)$  of the {\it  Hardy-Littlewood maximal operator} $\M$ and  the
 Calder\'on-Zygmund operator $\K$
$$
\M f(x) = \sup_{\B(x)}\int_{\B(x)} |f(y)|\, dy\,,\quad \K f(x) =P.V. \int_{\R^n}\frac{ f(y)}{|x-y|^n}\, dy\,.
$$
Integral operators of that kind appear in the representation formulae of the solutions of various  PDEs. Thus the continuity of
the Calder\'on-Zygmund integral  in certain  functional  space permit to study the regularity of the solutions of boundary value problems for linear PDEs in the corresponding space.

In \cite{Mi} Mizuhara extended the definition of $L_{p,\lambda}$ taking a {\it non-negative measurable function}
$\phi(x,r):\Rn\times\R_+\to\R_+$ instead of  {\it the Morrey weight} $r^\lambda$ in the definition of $L_{p,\lambda}.$  Precisely, 
$f\in L_{p,\phi}(\R^n)$ if $f\in L_p^{\loc}(\R^n),$ $p\in[1,\infty)$ and 
$$
\|f\|_{p,\phi}=\sup_{\B_r(x)}\left( \frac1{\phi(x,r)} \int_{\B_r(x)} |f(y)|^p\,dy \right)^{\frac1p}<\infty
$$
and the supremo is taken over all balls in $\R^n.$

  Later Nakai extended  the results of Chiarenza and Frasca to the case of
 $L_{p,\phi}.$ Imposing the next  integral and doubling conditions on $\phi$ (see \cite{Na})
\begin{align*}
 &\kappa_1^{-1}\leq \frac{\phi(x_0,t)}{\phi(x_0,r)}   \leq \kappa_1, \quad r \leq  t \leq 2r,\\
 &\int_r^\infty\frac{ \phi(x_0,t)}{t^{n+1}} dt  \le
 \kappa_2 \,\frac{ \phi(x_0,r)}{r^n}
\end{align*}
he proved  boundedness  of $\M$ and $\K$  
$$
\|\M f\|_{p,\phi}\leq C \| f\|_{p,\phi},\qquad    \|\K f\|_{p,\phi}\leq C \| f\|_{p,\phi}
$$
 for all  $ f\in L_{p,\phi}(\R^n),$ $p\geq 1.$

The next extension  of the Morrey spaces is given by  the first author. He   defined  generalized Morrey spaces $M_{p,\varphi}$ with normalized norm under more general condition on the  weight $\varphi:\R^n\times\R_+\to \R_+$ and considered continuity of various classical integral operators from one space $M_{p,\varphi_1}$ to another $M_{p,\varphi_2}$ under suitable condition on the pair $(\varphi_1, \varphi_2).$ In
\cite{GulDoc} (see also \cite{GulBook, GulJIA})  it  is shown that if
\begin{equation}\label{GulZSIO}
\int_r^{\infty} \varphi_1(x,t)\frac{dt}{t} \le  C \,\varphi_2(x,r) 
\end{equation}
then the  operator $\K$ is bounded from $M_{p,\varphi_1}$ to $M_{p,\varphi_2}$ for $p > 1$ and
from $M_{1,\varphi_1}$ to the weak space  $WM_{1,\varphi_2}$.
 In \cite{AkbGulMus1, GULAKShIEOT2012}, Guliyev et al.  introduced a weaker condition on the pair $(\varphi_1,\varphi_2)$
 under which boundedness of  the classical integral operators  from $M_{p,\varphi_1}$ to $M_{p,\varphi_2}$ is proved. Precisely,   if
\begin{equation}\label{wueq2}
\int_{r}^{\infty}\frac{\es_{t<s<\infty}\varphi_{1}(x,s)s^{\frac{n}{p}}}{t^{\frac{n}{p}}}\, \frac{dt}{t}\leq C\,\varphi_{2}(x,r),
\end{equation}
then $\K$ is bounded from $M_{p,\varphi_1}$   to another $M_{p,\varphi_2}$ for  $p>1$ and from $M_{1,\varphi_1}$ to the weak space $WM_{1,\varphi_2}.$     Let us note that the condition  \eqref{GulZSIO}  describes wider class of weight functions than  \eqref{wueq2} (see  \cite{GulEMJ2012}).

For more recent  results on boundedness and continuity of singular integral operators in  generalized Morrey
and new functional spaces  and their application in   the  theory of the differential equations  see
\cite{AkbGulMus1, GulJIA, GS, GuHaSam,  GulSoft1, GulSoft2,   Pal,  Sf1, Sf2} and the references therein.

Consider now the weighted  $L_p$-spaces  $L_{p,w}$ consisting of measurable functions $f$ for which
$$
\|f\|_{p,w}=\left(\int_{\R^n}|f(y)|^pw(y)\,dy \right)^{\frac1p}\,.
$$

 In   \cite{Muckenh}  Muckenhoupt  showed  that the well known maximal inequality holds in   $L_{p,w}$  if and only if the weight $w$ satisfies certain integral condition called {\it $A_p$-condition}.  Later,
 Coifman and Fefferman \cite{CoiFeff}   studied the continuity of some classical singular  integrals in the {\it Muckenhoupt spaces} (see also \cite{MuckWh1, MuckWh}).

 Recently, Komori and Shirai \cite{KomShir}  defined the weighted Morrey spaces $L_{p,\kappa}(w)$  endowed by the norm
$$
\|f\|_{p,w,k}= \sup_{\B}\left( \frac1{w(\B)^k} \int_\B  |f(y)|^pw(y)\,dy\right)^\frac1p\,.
$$
They  studied the boundedness of the Calder\'on-Zygmund operator   $\K$   in these spaces.   A natural extension of their  results are 
the generalized  weighted Morrey spaces  $M_{p,\varphi}(w)$ with $w\in A_p $ and $\varphi$ satisfying  \eqref{GulZSIO}. 
 In  \cite{GulEMJ2012}    (see also  \cite{GulKarMustSer, KarGulSer}) it is proved  boundedness in $M_{p,\varphi}(w)$  of sub-linear operators generated by classical operators as $\M,$ $\K,$ the Riesz potential and others, covering such way the results obtained in \cite{Na} and  \cite{KomShir}. Our goal here is to obtain a priori estimate for the solution of the Dirichlet problem for linear elliptic equations in these spaces.

The paper is organized as follows.
We begin  introducing  the  functional spaces that we are going to use.
In 
 Sections~\ref{sec3a} and~\ref{sec3} we study continuity in the spaces $M_{p,\varphi}(w)$ of certain sub-linear integrals and their commutators with  functions with bounded mean oscillation.  These results permit to obtain continuity of the Calder\'on-Zygmund operator,  with bounded functions and some nonsingular integrals which is done in Section~\ref{sec5}. The last section is dedicated  to the    Dirichlet problem for  linear  elliptic equations with discontinuous coefficients.  
This problem is firstly studied by  Chiarenza, Frasca and Longo. In their pioneer works \cite{ChFraL1, ChFraL2} they prove unique  strong solvability of
\begin{equation} \label{LDP}
\begin{cases}
    \L u \equiv a^{ij}(x)D_{ij} u=f(x) \quad \text{ a.a.  } x\in  \Omega,\\
    u\in\ W^2_p(\Omega)\cap \overset{\circ}{W}{}^1_p(\Omega), \   p\in(1,\infty), \ a^{ij}\in VMO
  \end{cases}
\end{equation}
extending  this way the classical theory of operators with continuous coefficients to those   with discontinuous coefficients.
Later their results  have been   extended in the Sobolev-Morrey spaces
$W^2_{p,\lambda}(\Omega)\cap  \overset{\circ}{W}{}^1_p(\Omega),$   $\lambda\in(1,n)$   (see \cite{DPR})  and the generalized Sobolev-Morrey spaces
$W^2_{p,\phi}(\Omega)\cap \overset{\circ}{W}{}^1_p(\Omega)$  (see  \cite{Sf2})  with $\phi$ as in \cite{Na}.
  In \cite{GulSoft1}  we have studied  the regularity of the solution  of \eqref{LDP}  in generalized Sobolev-Morrey  spaces $W^2_{p,\varphi}(\Omega)$ where the weight function $\varphi$  satisfies a certain  supremal condition derived from 
\eqref{wueq2}.
 We show that   $\L u\in M_{p,\varphi}(\Omega)$ implies  $D_{ij}u\in M_{p,\varphi}(\Omega)$ satisfying the estimate
$$
\|D^2u\|_{p,\varphi;\Omega}\leq C\big(\|\L  u\|_{p,\varphi;\Omega} + \|u\|_{p,\varphi;\Omega}  \big)\,.
$$
These studies are extended on divergence form  elliptic/parabolic equations in \cite{BS, GulSoft3}.

In this paper we use the following notions:
\begin{align*}
&D_iu=\partial u/ \partial x_i, \  Du=(D_1u,\ldots,D_nu) \text{ means the gradient of  } u,\\
&   D_{ij}u= \partial^2 u/\partial x_i\partial x_j, \ D^2u=\{D_{ij}u\}_{ij=1}^n \text{ means the Hessian matrix of } u,\\
&\B_r(x_0)=\{x\in{\R}^n:\ |x-x_0|<r  \} \text{ is a ball  centered at a fixed point } x_0\in{\R}^n,\\
&\B_r(x)\equiv\B_r\equiv \B \text{ is a ball centered at any point } x\in {\R}^n, \   |{\B}_r|=C r^n,\\
&\B_r^c=\R^n\sm \B_r,\quad 2{\B}_r= \B_{2r},\\
&{\SS}^{n-1}=\{y\in  {\R}^n:\  |y-x|=1 \} \text{ is a unit sphere at } {\R}^n \text{ centered in } x\in \R^n,\\
&\R^n_+=\{x\in\Rn:\  x_n>0\}\,.
\end{align*}
For any measurable set $A$ and $f\in L_p(A),$ $1<p<\infty$ we write
$$
\|f\|_{L_p(A)}= \|f\|_{p;A}=\left( \int_A|f(y)|^p\,dy  \right)^{\frac1p}, \quad \|\cdot\|_{p;\R^n}\equiv \|\cdot\|_p\,.
$$
 The standard summation convention on repeated upper and lower indices is  adopted. The letter $C$ is used for various positive  constants and may change from one
occurrence to another.

\section{Weighted spaces }\label{sec2}
\setcounter{theorem}{0}
\setcounter{definition}{0}
\setcounter{lemma}{0}
\setcounter{corollary}{0}

We start with the definitions of some function   spaces that we are going to use.
\begin{definition}(see \cite{JN, Sar})
Let  $a\in L_1^{\rm loc}({\R}^n)$ and $a_{{\B}_r}= \frac{1}{|{\B}_r|}\int_{{\B}_r} a(x)\, dx. $   Define
$$
\gamma_a(R)= \sup_{ r\leq R}\frac1{|{\B}_r|}\int_{{\B}_r}|a(y)-a_{{\B}_r}|\,dy\qquad \forall \  R>0.
$$
 We say that $a\in BMO$  ({\it bounded mean oscillation})  if
$$
\|a\|_{\ast}=\sup_{ R>0}\gamma_a(R)<+\infty.
$$
The quantity $\|a\|_\ast$ is a norm in $BMO$ modulo constant functions under which $BMO$ is a Banach space.
If
$$
\lim_{R\to 0}\gamma_a(R)=0
 $$
then  $a\in VMO$   ({\it vanishing mean oscillation})  and we call   $\gamma_a(R)$ a   $VMO$-modulus of $a.$

For any bounded domain $\Omega\subset {\R}^n$ we define
 $BMO(\Omega)$ and $VMO(\Omega)$ taking  $a\in L_1(\Omega)$ and  integrating over  $\Omega_r=\Omega\cap\B_r. $
\end{definition}
According to \cite{A}, having a function $a\in BMO(\Omega)$ or  $VMO(\Omega)$ it is possible to extend it in the whole  space  preserving its
 $BMO$-norm or $VMO$-modulus, respectively. In the following we use this  extension  without explicit references.
\begin{lemma}{\rm(John-Nirenberg lemma, \cite{JN})}\label{lem2.4.}
Let $a\in BMO$ and $p\in (1,\i)$. Then for any ball $\B$  holds
$$
\left( \frac{1}{|\B|}\int_{\B}|a(y)-a_{\B}|^p dy\right)^{\frac{1}{p}}
\leq C(p) \|a\|_{*} .
$$
\end{lemma}
As an immediate consequence of Lemma~\ref{lem2.4.} we get the next property.
\begin{corollary}
 Let $a\in BMO$ then for all $0<2r<t$  holds
\begin{equation}\label{propBMO}
\big|a_{\B_r}-a_{\B_t}\big| \le C \|a\|_\ast \ln \frac{t}{r}
\end{equation}
where the constant is independent of $a,x,t$ and $r.$
\end{corollary}

We call {\it weight} a non-negative locally integrable function  on  $\R^n.$  Given a weight $w$ and a measurable set $\E$
we denote the $w$-measure of $\E$ by
$$
w(\E)=\int_\E w(x)\,dx\,.
$$
Denote by $L_{p,w}(\R^n) $ or $L_{p,w}$ the weighted $L_p$ spaces. It turns out that the strong type $(p,p)$ inequality
$$
\left(\int_{\R^n}(\M f(x))^pw(x)\,dx \right)^{\frac1p}   \leq C_p   \left(\int_{\R^n}|f(x)|^pw(x)\, dx  \right)^{\frac1p}
$$
holds for all $f\in L_{p,w}$ if and only if the weight function  satisfies the {\it Muckenhoupt}  $A_p$-condition
\begin{equation}\label{Ap}
[w]_{A_p}:  =  \sup_{\B}
\left(\frac{1}{|\B|}\int_{\B} w(x)\,dx\right)\left(\frac{1}{|\B|}\int_{\B}
w(x)^{-\frac1{p-1}}\,dx\right)^{p-1}<\infty\,.
\end{equation}
The expression $[w]_{A_p}$  is called {\it characteristic constant} of $w.$
The function $w$ is $A_1$ weight if $\M w(x)\leq C_1 w(x)$ for almost all  $x\in \R^n.$ The minimal constant $C_1$ for which the inequality holds is  the $A_1$ {\it characteristic constant } of $w.$

We summarize some basic  properties of the $A_p$ weights in the next  lemma (see \cite{Grafakos, Muckenh} for more details).

\begin{lemma}
   (1)  \   Let  $w\in A_p$ for $1\leq  p<\infty.$  Then for each $\B$
\begin{equation}  \label{norma-w}
1\leq [w]_{A_{p}(\B)}^{\frac1p}=|\B|^{-1} \|w \|_{L_1(\B)}^{\frac1p} \, \|w^{-\frac1p} \|_{L_{p'}(\B)} \leq [w]_{A_p}^{\frac1p}\,.
\end{equation}

(2) The function $w^{-\frac1{p-1}}$ is in $A_{p'}$  where $\frac1p+\frac1{p'}=1,$ $1<p<\infty$ with characteristic constant
$$
[w^{-\frac1{p-1}}]_{A_{p'}}=[w]_{A_p}^{\frac1{p-1}}\,.
$$

(3) \   The classes  $A_p$ are increasing as $p$ increases and
$$
[w]_{A_q}\leq [w]_{A_p}, \qquad  1\leq q<p<\infty\,.
$$

(4) \ The measure $w(x)dx$ is doubling, precisely,  for all $\lambda>1$
$$
w(\lambda \B) \le \lambda^{np}[w]_{A_p} w(\B)\,.
$$

(5)  \    If $w \in A_p$ for some $1 \le p \le \infty$, then there exist $C>0$ and $\delta > 0$
such that for any ball $\B$ and a measurable set $\E \subset \B$,
$$
\frac{1}{[w]_{A_p}}\left( \frac{|\E|}{|\B|} \right)\leq\frac{w(\E)}{w(\B)} \le C \left( \frac{|\E|}{|\B|} \right)^{\delta}.
$$

(6)    For each $1\leq p<\infty$ we have
$$
\bigcup_{1\leq p<\infty}A_p=A_\infty\quad \text{ and  } \quad
[w]_{A_\infty}\leq [w]_{A_p}\,.
$$

(7)  For each $a\in BMO,$ $1\leq p<\infty$  and $w\in A_\infty$ we have
\begin{equation}  \label{weightBMO}
\|a\|_{\ast}  =C \sup_{\B}\left(\frac{1}{w(\B)}
\int_{\B}|a(y)-a_{\B}|^p w(y)\,dy\right)^{\frac{1}{p}}\,.
\end{equation}
\end{lemma}
The next result follows from \cite[Lemma~4.4]{GulEMJ2012}.
\begin{lemma}
 Let $w \in A_p$  with $1<p<\infty$ and $a\in BMO.$  Then
\begin{equation}\label{EMJ}
\Big( \frac{1}{w^{1-p'}(\B)} \int_{\B} |a(y)-a_{\B}|^{p'} w(y)^{1-p'}\, dy\Big)^{\frac{1}{p'}}
\le C [w]_{A_p}^{\frac{1}{p}}  \| a \|_{\ast},
\end{equation}
where $C$ is independent of $a$, $w$ and  $\B.$
\end{lemma}
\begin{definition} \label{def1}
Let $\varphi(x,r)$ be weight in  $\R^n\times\R_+\to \R_+$ and  $w\in A_p,$  $ p\in[1,  \infty).$ The generalized weighted Morrey space
$M_{p,\varphi}(\R^n,w)$ or $ M_{p,\varphi}(w)$ consists of all
functions $f\in L_{p,w}^{\rm loc}(\R^n)$ such that
$$
\|f\|_{p,\varphi,w} = \sup_{x\in\R^n, r>0}
\varphi(x,r)^{-1} \left(w(\B_r(x))^{-1} \, \int_{\B_r(x)}|f(y)|^p w(y)\, dy\right)^{\frac1p}<\infty\,.
$$
For any bounded domain $\Omega$ we define $M_{p,\varphi}(\Omega,w)$  taking  $f\in L_{p,w}(\Omega)$ and  integrating over  $\Omega_r=\Omega \cap \B_r(x), x\in \Omega.$

Generalized Sobolev-Morrey  space $W^2_{p,\varphi}(\Omega,w)$  consists of all   functions $u\in W^2_{p.w}(\Omega)$
with distributional derivatives $D^s u\in M_{p,\varphi}(\Omega,w)$,  $0\leq |s|\leq 2$ endowed by the norm
$$
\|u\|_{W^2_{p,\varphi}(\Omega,w)}=\sum_{0\leq |s|\leq 2}\|D^s f\|_{p,\varphi,w;\Omega}.
$$
The space $ W^2_{p,\varphi}(\Omega,w)\cap \overset{\circ}{W}{}^1_{p}(\Omega,w)$  consists
of all functions $u\in W^2_{p,w}(\Omega)\cap \overset{\circ}{W}{}^1_{p,w}(\Omega)$ with $D^su\in M_{p,\varphi}(\Omega,w),$ $0\leq |s|\leq 2$
and is    endowed by the same norm.
Recall that $\overset{\circ}{W}{}^1_{p,w}(\Omega)$   is the closure of $C_0^\infty(\Omega)$ with respect to the norm in $W^1_{p,w}(\Omega).$
\end{definition}
\begin{remark}\label{rem1}\em
The density of the $C_0^\infty$ functions in the weighted Lebesgue space $L_{p,w}$ is proved in
\cite[Chapter~3, Theorem~3.11]{Sob}.
\end{remark}

\section{Sublinear operators   generated by singular integrals  in $M_{p,\varphi}(w)$}\label{sec3a}

 Let  $T$ be a  sub-linear operator. Suppose  that $T$ satisfy 
\begin{equation}\label{subl1}
|Tf(x)|\le C \int_{\R^n}
\frac{|f(y)|}{|x-y|^{n}} \,dy
\end{equation}
  for any
$f\in L_1(\R^n)$ with compact support and $x\notin \supp f.$ 

The next  results   generalize some estimates obtained in  \cite{GulDoc, GulJIA, GULAKShIEOT2012, GulKarMustSer,  KarGulSer}.  The proof is  as in \cite{GULAKShIEOT2012} and  makes use of the  boundedness of the weighted Hardy operator
$$
H^*_\psi g(r):= \int_r^\infty\, g(t)\psi(t)\, dt, \qquad 0<r<\infty\,.
$$
\begin{theorem}(\cite{GulAJM2013, GulJMS2013})\label{Hardy}
Suppose that $v_1, v_2,$ and $\psi$ are weights on $\R_+.$  Then the inequality
\begin{equation}\label{eqH1}
\ess_{r>0} v_2(r)H^*_\psi g(r) \leq C \ess_{r>0 } v_1(r) g(r)
\end{equation}
holds with some $C>0$ for all non-negative and nondecreasing $g$ on $\R_+$ if and only if
\begin{equation}\label{eqH2}
B:=\ess_{r>0} v_2(r)\int_r^\infty\,  \frac{\psi (t)}{\ess_{t<s<\infty}v_1(s)}\, dt<\infty
\end{equation}
and $C=B$ is the best constant in  \eqref{eqH1}.
\end{theorem}

\begin{theorem}\label{Tbound}
Let $1< p<\infty$, $w \in A_{p}$ and the pair  $(\varphi_1,\varphi_2)$
satisfy
\begin{equation}\label{condition1}
\int_{r}^{\infty} \frac{\es_{t<s<\infty} \varphi_1(x,s)
w(\B_s(x))^{\frac{1}{p}}}{w(\B_t(x))^{\frac{1}{p}}} \, \frac{dt}{t} \le
 C \, \varphi_2(x,r),
\end{equation}
and  $T$ be a sub-linear operator satisfying   \eqref{subl1}.  If $T$ is   bounded on $L_{p,w}$
and  $\|Tf\|_{p,w}\leq C [w]_{A_p}^{\frac1p} \|f\|_{p,w}$, then $T$ is bounded from $M_{p,\varphi_1}(w)$ to $M_{p,\varphi_2}(w)$ and
\begin{equation}\label{eqTbound}
\|T f\|_{p,\varphi_2,w}\leq C [w]_{A_p}^{\frac{1}{p}} \|f\|_{p,\varphi_1,w}
\end{equation}
with a constant independent of $f.$
\end{theorem}

For any $a\in BMO$ consider the commutator $T_af=aTf -T(af).$  Let $T_a$ be a sub-linear operator satisfying 
\begin{equation}\label{subl2}
|T_a f(x)|\leq C \int_{\R^n}|a(x)-a(y)| \,  \frac{|f(y)|}{|x-y|^n}\, dy
\end{equation}
 for any $f\in L_1(\R^n)$ with a compact support and
$x\not\in \supp f.$   Suppose in addition that $T_a$ is bounded in $L_{p,w}$  and 
satisfies   $\|T_af\|_{p,w} \leq C \|a\|_* [w]_{A_p}^{\frac1p} \|f\|_{p,w}.$  Then the next  result is valid and the prood is  as in \cite{GULAKShIEOT2012}, making use  of Theorem~\ref{Hardy}.
\begin{theorem}\label{TAbound}
Let $p\in(1,\infty),$ $w\in A_p,$ $a\in BMO$  and the pair  $(\varphi_1, \varphi_2)$ satisfy
\begin{equation}\label{condition2}
\int_{r}^{\infty}\left(1+\ln\frac{t}r\right) \frac{\es_{t<s<\infty} \varphi_1(x,s)
w(\B_s(x))^{\frac{1}{p}}}{w(\B_t(x))^{\frac{1}{p}}} \, \frac{dt}{t} \le
 C \, \varphi_2(x,r)
\end{equation}
with a constant independent on $x$ and $r.$ Suppose that $T_a$  is bounded in $L_{p,w}$ and  satisfies  \eqref{subl2}. 
Then $T_a$ is bounded  from $M_{p,\varphi_1}(w)$ to $M_{p,\varphi_2}(w)$ and
\begin{equation}\label{eqTa_bound}
\|T_a f\|_{p,\varphi_2,w}\leq C [w]_{A_p}^{\frac{1}{p}} \|a\|_* \|f\|_{p,\varphi_1,w}\,.
\end{equation}
\end{theorem}

\section{ Sublinear  operators generated by nonsingular integrals in  $M_{p,\varphi}(w)$}\label{sec3}
\setcounter{theorem}{0}
\setcounter{definition}{0}
\setcounter{lemma}{0}
\setcounter{corollary}{0}

For any $x\in \R^n_+$ define  $\tl x=(x_1,\ldots,x_{n-1},-x_n).$
Let $\wtl{T}$ be a sub-linear operator with a nonsingular kernel. Suppose that $\wtl{T}$ satisfy the condition 
\begin{equation}\label{tlT}
|\wtl Tf(x)|\leq C \int_{\R^n_+}\frac{|f(y)|}{|\tl x-y|^n}\,dy
\end{equation}
 for any $f\in L_1(\R^n_+)$
with a compact support. 
\begin{lemma} \label{theorXS} \label{lem3.3.}
Let $w \in A_{p}$, $p\in(1,\infty)$, the operator $\wtl T$  satisfy   \eqref{tlT} and $\wtl T$ is bounded on $L_{p,w}(\R^n_+)$. Let also for any  fixed  $x_0\in \R^n_+$ and for any  $f\in L_{p,w}^{\rm loc}(\R^n_+)$
\begin{equation}\label{eq3.5.nm}
\int_r^\infty   w(\B_t^+(x_0))^{-\frac1p}  \|f\|_{p,w;\B_t^+(x_0)} \, \frac{dt}{t}<\i \,.
\end{equation}
Then
\begin{equation}\label{eq3.5.}
\|\wtl T f\|_{p,w;\B_r^+(x_0)} \le C [w]_{A_p}^{\frac{1}{p}} \, 
w(\B_r^+(x_0))^{\frac{1}{p}} \int_{2r}^{\i} w(\B_t^+(x_0))^{-\frac{1}{p}} \|f\|_{p,w;\B_t^+(x_0)}  \, \frac{dt}{t}
\end{equation}
with a  constant independent of  $x_0$, $r$, and $f$.
\end{lemma}

\begin{proof}
Consider the decomposition
 $f=f_1+f_2$ with
$f_1=f\chi_{2\B_r^+(x_0)}$ and $f_2=f\chi_{(2\B_r^+(x_0))^c}$.
Because of the boundedness of $\wtl T$ in  $L_{p,w}(\R^n_+)$  we have
as in \cite{GulSoft1}
\begin{equation*}
\|\wtl T f_1\|_{p,w;\B_r^+(x_0)}\leq C [w]_{A_p}^{\frac{1}{p}} \|f\|_{p,w;2\B_r^+(x_0)}\,.
\end{equation*}
Since for any   $\tl x \in \B_r^+(x_0)$ and  $y\in  (2\B_r^+(x_0))^c$  it holds
\begin{equation}\label{xy}
\frac{1}{2}|x_0-y|\le |\tl x-y| \le \frac{3}{2}|x_0-y|.
\end{equation}
we get as in \cite{GulSoft1}
$$
|\wtl T f_2(x)|\leq  C \int_{2r}^\i \left(\int_{\B_t^+(x_0)}|f(y)|dy  \right)\frac{dt}{t^{n+1}}.
$$
Making use of the  H\"older  inequality and \eqref{norma-w}  we get
\begin{equation} \label{sal00}
\begin{split}
|\wtl T f_2(x)| & \leq  C \int_{2r}^{\i}\|f\|_{p,w;\B_t^+(x_0)} \, \|w^{-\frac1p}\|_{p';\B_t^+(x_0)} \, \frac{dt}{t^{n+1}}
\\
& \le C [w]_{A_{p}}^{\frac1p} \, \int_{2r}^{\i}  w(\B_t^+(x_0))^{-\frac1p} \|f\|_{p,w;\B_t^+(x_0)} \, \frac{dt}{t}.
\end{split}
\end{equation}
Direct calculations give
\begin{equation}
\|\wtl T f_2\|_{p,w;\B_r^+(x_0)}  \leq  C [w]_{A_{p}}^{\frac1p} \, w(\B_r^+(x_0))^{\frac1p}  \int_{2r}^{\i}\frac{\|f\|_{p,w;\B_t^+(x_0)}}{ w(\B_t^+(x_0))^{\frac1p}}\,  \frac{dt}{t}
\end{equation}
for all $f\in L_{p,w}(\R^n_+)$   satisfying \eqref{eq3.5.nm}.
Thus,
\begin{align}\label{Tf}
\nonumber
\|\wtl T f\|_{p,w;\B_r^+(x_0)}  &   \leq  \|\wtl T f_1\|_{p,w;\B_r^+(x_0)}+\|\wtl Tf_2\|_{p,w;\B_r^+(x_0)} \\
& \leq C  [w]_{A_{p}}^{\frac1p}  \|f\|_{p,w;2\B_r^+(x_0)}\\
\nonumber
& +  C [w]_{A_{p}}^{\frac1p} w(\B_r^+(x_0))^{\frac1p} \, \int_{2r}^{\i}\frac{\|f\|_{p,w;\B_t^+(x_0)} }{   w(\B_t^+(x_0))^{\frac1p}}  \,\frac{dt}{t}\, .
\end{align}

On the other hand, by \eqref{norma-w}
\begin{align}  \label{sal01}
\nonumber
\|f\|_{p,w;2\B_r^+(x_0)} & \le C |\B_r^+(x_0)| \|f\|_{p,w;2\B_r^+(x_0)}
\int_{2r}^{\i}\,\frac{dt}{t^{n+1}} \\
\nonumber
& \le C |\B_r^+(x_0)| \, \int_{2r}^{\i}\|f\|_{p,w;\B_t^+(x_0)}\,\frac{dt}{t^{n+1}} \\
\nonumber
& \le C[w]_{A_p}^{-\frac1p} w(\B_r^+(x_0))^{\frac1p} \, \int_{2r}^{\i}\|f\|_{p,w;\B_t^+(x_0)} \, \|w^{-\frac1p}\|_{p';\B_t^+(x_0)} \, \frac{dt}{t^{n+1}} \\
\nonumber
& \le C [w]_{A_{p}}^{-\frac1p}  w(\B_r^+(x_0))^{\frac1p} \, \int_{2r}^{\i} [w]_{A_p}^{\frac1p} w(\B_t^+(x_0))^{-\frac1p}  \|f\|_{p,w;\B_t^+(x_0)}  \, \frac{dt}{t}\\
& \leq  w(\B_r^+(x_0))^{\frac1p} \, \int_{2r}^{\i}  w(\B_t^+(x_0))^{-\frac1p}  \|f\|_{p,w;\B_t^+(x_0)}  \, \frac{dt}{t}
\end{align}
which unified with \eqref{Tf} gives \eqref{eq3.5.}.
\end{proof}

\begin{theorem}\label{3.4.}
Suppose that   $w \in A_{p},$  $p\in(1,\infty),$  the pair  $(\varphi_1,\varphi_2)$   satisfies the condition \eqref{condition1} for any $x\in \R^n_+$  and \eqref{tlT} holds. Then if  $\wtl T$ is bounded in $L_{p,w}(\R^n_+),$  then it is bounded from
$M_{p,\varphi_1}(\R^n_+,w)$ in $M_{p,\varphi_2}(\R^n_+,w)$ and
\begin{equation}\label{normTf}
\|\wtl T f\|_{p,\varphi_2,w;\R^n_+} \leq C [w]_{A_p}^{\frac{1}{p}} \|f\|_{p,\varphi_1,w;\R^n_+}
\end{equation}
with a constant independent of $f.$
\end{theorem}
\begin{proof}
 By  Lemma~\ref{lem3.3.}  we have
$$
\|\wtl T f\|_{p,\varphi_2,w;\R^n_+}  \leq C [w]_{A_p}^{\frac{1}{p}} \sup_{x\in\R^n_+,\,r>0}\varphi_2(x,r)^{-1} \int_r^{\i}  w(\B_t^+(x))^{-\frac1p} \|f\|_{p,w;\B_t^+(x)}  \, \frac{dt}{t}\,.
$$
Applying the Theorem~\ref{Hardy} with
$$
 v_1(r)=\varphi_1(x,r)^{-1}  w(\B_r^+(x))^{-\frac1p},\qquad v_2(r)= \varphi_2(x,r)^{-1},
$$
$$
\psi(r)= w(\B_r^+(x))^{-\frac1p } r^{-1},\qquad    g(r)= \|f\|_{p,w;\B_r^+(x)}
$$
to the above integral, we get  as in \cite{GulSoft1}
\begin{align*}
\|\wtl T f\|_{p,\varphi_2,w;\R^n_+}&  \leq  C [w]_{A_p}^{\frac{1}{p}}  \sup_{x\in\R^n_+, r>0}\varphi_1(x,r)^{-1} \, w(\B_r^+(x))^{-\frac1p} \,
\|f\|_{p,w;\B_r^+(x)}\\
& = C [w]_{A_p}^{\frac{1}{p}} \|f\|_{p,\varphi_1,w;\R^n_+}.
\end{align*}
\end{proof}

\section{ Commutators of sub-linear  operators generated by nonsingular integrals  in  $M_{p,\varphi}(w)$}\label{sec4}
\setcounter{theorem}{0}
\setcounter{definition}{0}
\setcounter{lemma}{0}
\setcounter{corollary}{0}

For any  $a\in BMO$ consider the  commutator  ${\wtl T}_{a}f= a\wtl T f-\wtl T(af)$
where $\wtl T$ is the nonsingular operator satisfying \eqref{tlT} and $f\in L_1(\R^n_+)$ with a compact support. Suppose that for  $x\notin supp f$
\begin{equation}\label{sublcomm}
|{\wtl T}_{a}f(x)|\le C \int_{\R^n_+} |a(x)-a(y)|\, \frac{|f(y)|}{ |\tl x-y|^n }\, dy,
\end{equation}
where $C$ is independent of $f, a,$   and $x$.

Suppose in addition that ${\wtl T}_a$ is bounded in $L_{p,w}(\R^n_+),$ $w\in A_p,$  $  p\in(1,\infty)$ satisfying the estimate
$\|{\wtl T}_af\|_{p,w;\R^n_+}\leq C \, [w]_{A_{p}}^{\frac1p} \, \|a\|_{\ast} \, \|f\|_{p,w;\R^n_+}$.
 Our aim is to show boundedness of ${\wtl T}_a$ in $M_{p,\varphi}(\R^n_+,w)$.

To estimate the commutator we shall employ the same idea which we
used in the proof of Lemma~\ref{lem3.3.} (see \cite{GulSoft1} for details).
\begin{lemma}\label{lem5.1.}
Let  $w \in A_p,$  $p\in(1,\infty),$ $a \in BMO$ and  ${\wtl T}_{a}$ be a bounded  operator in $L_{p,w}(\R^n_+)$ satisfying
 \eqref{sublcomm} and the estimate
$\|{\wtl T}_{a} f\|_{p,w;\R^n_+}  \leq C  [w]_{A_{p}}^{\frac1p} \|a\|_* \, \|f\|_{p,w;\R^n_+} .$   Suppose that for all
$f\in\Lpwloc+$,  $x_0\in\R^n_+$ and $r>0$ applies  the next  condition 
\begin{equation}\label{condition3}
\int_r^{\i} \Big(1+\ln\frac{t}{r}\Big) \frac{\|f\|_{p,w;\B_t^+(x_0)}}{ w(\B_t^+(x_0))^{\frac1p}} \, \frac{dt}{t}<\i\,.
\end{equation}
Then
\begin{equation}\label{norm-Ta}
 \|{\wtl T}_{a} f\|_{p,w;\B_r^+(x_0)}  \leq C [w]_{A_p}^{\frac{1}{p}} \|a\|_{*} \, w(\B_r^+(x_0))^{\frac1p} \int_{2r}^{\i}  \Big(1+\ln \frac{t}{r}\Big)
\frac{ \|f\|_{p,w;\B_t^+(x_0)}}{ w(\B_t^+(x_0))^{\frac1p}} \, \frac{dt}{t}\,.
\end{equation}
\end{lemma}
\begin{proof}
The decomposition  $f= f\chi_{2\B_r^+(x_0)}+  f\chi_{(2\B_r^+(x_0))^c}=  f_1+f_2$  gives
$$
\|{\wtl T}_{a} f\|_{p,w;\B_r^+(x_0)}\leq  \|{\wtl T}_{a} f_1\|_{p,w;\B_r^+(x_0)}+\|{\wtl T}_{a} f_2\|_{p,w;\B_r^+(x_0)}.
$$
From the boundedness of ${\wtl T}_{a}$ in $L_{p,w}(\R^n_+)$ it follows
$$
\|{\wtl T}_{a} f_1\|_{p,w;\B_r^+(x_0)}  \leq C  [w]_{A_{p}}^{\frac1p} \,\|a\|_{*} \, \|f\|_{p,w;2\B_r^+(x_0)}.
$$
On the other hand, because of  \eqref{xy} we  can write
\begin{equation*}
\begin{split}
&\|{\wtl T}_{a} f_2\|_{p,w;\B_r^+(x_0)}\\
&\   \leq C \left(\int_{\B_r^+(x_0)}\left(\int_{(2\B_r^+(x_0))^c}
\frac{|a(y)-a_{\B^+_r(x_0)}||f(y)|}{|x_0-y|^{n}}\, dy\right)^p w(x)\, dx\right)^{\frac1p}\\
&\  + C \left(\int_{\B_r^+(x_0)}\left(\int_{(2\B_r^+(x_0))^c}\frac{|a(x)-a_{\B_r^+(x_0)}||f(y)|}{|x_0-y|^{n}}\, dy\right)^p w(x)\, dx\right)^{\frac1p}\\
&\   =I_1+I_2.
\end{split}
\end{equation*}
Where, as in \cite{GulSoft1},  we have
$$
I_1 \leq  C w(\B_r^+(x_0))^{\frac1p}\int_{2r}^{\i}\int_{\B_t^+(x_0)}|a(y)-a_{\B_r^+(x_0)}||f(y)|\,dy\,\frac{dt}{t^{n+1}}\,.
$$
Applying H\"older's inequality, Lemma~\ref{lem2.4.},  \eqref{propBMO} and \eqref{EMJ}, we get
\begin{align*}
I_1  & \leq C w(\B_r^+(x_0))^{\frac1p} \int_{2r}^{\i}\int_{\B_t^+(x_0)} |a(y)-a_{\B_t^+(x_0)}||f(y)|\,dy\,\frac{dt}{t^{n+1}}
\\
&+  C w(\B_r^+(x_0))^{\frac1p} \int_{2r}^{\i}\int_{\B_t^+(x_0)} |a_{\B_t^+(x_0)}-a_{\B_r^+(x_0)}||f(y)|\,dy\,\frac{dt}{t^{n+1}}
\\
& \leq C \, w(\B_r^+(x_0))^{\frac1p}\int_{2r}^{\i}
 \left(\int_{\B_t^+(x_0)}|a(y)-a_{\B_t^+(x_0)}|^{p'} w(y)^{1-p'}\,dy\right)^{\frac{1}{p'}}\\
&   \phantom{hhhhhhhhhhhhhhhhhhhhhhhhhhhhhh} \times  \|f\|_{p,w;\B_t^+(x_0)}\,\frac{dt}{t^{n+1}}
\\
& + C [w]_{A_p}^{\frac{1}{p}} w(\B_r^+(x_0))^{\frac1{p}} \|a\|_*  \int_{2r}^{\i} \ln\frac{t}{r} \|f\|_{ p,w;\B_t(x_0)}
w(\B_t(x_0))^{-\frac{1}{p}} \frac{dt}{t}
\\
& \leq C [w]_{A_p}^{\frac{1}{p}}  w(\B_r^+(x_0))^{\frac{1}{p}} \|a\|_{*}  \int_{2r}^{\i}  \|f\|_{p,w;\B_t^+(x_0)}
 w(\B_t^+(x_0))^{-\frac1{p}}\,\frac{dt}{t^{n+1}}
\\
& + C [w]_{A_p}^{\frac{1}{p}}  w(\B_r^+(x_0))^{\frac{1}{p}} \|a\|_{*}  \int_{2r}^{\i}\ln \frac{t}{r} \|f\|_{p,w;\B_t^+(x_0)}
 w(\B_t^+(x_0))^{-\frac{1}{p}} \, \frac{dt}{t}
\\
& \leq C [w]_{A_p}^{\frac{1}{p}} w(\B_r^+(x_0))^{\frac{1}{p}} \|a\|_{*}  \int_{2r}^{\i}\Big(1+\ln \frac{t}{r}\Big) \|f\|_{p,w;\B_t^+(x_0)}
 w(\B_t^+(x_0))^{-\frac{1}{p}} \, \frac{dt}{t}\,.
\end{align*}
By Lemma~\ref{lem2.4.} and  \eqref{sal00}  we get
$$
I_2\leq  C [w]_{A_p}^{\frac{1}{p}} 
  \|a\|_{*} \, w(\B_r^+(x_0))^{\frac{1}{p}} \, \int_{2r}^{\i}  w(\B_t^+(x_0))^{-\frac{1}{p}}  \|f\|_{p,w;\B_t^+(x_0)}\, \frac{dt}{t}\,.
$$
Summing up $I_1$ and $I_2$  we get that for all $p \in (1,\infty)$
\begin{equation} \label{deckfV}
\|{\wtl T}_{a} f_2\|_{p,w;\B_r^+(x_0)} \leq C [w]_{A_p}^{\frac{1}{p}} \|a\|_{*} \, w(\B_r^+(x_0))^{\frac{1}{p}} \,
\int_{2r}^{\i} \Big(1+\ln \frac{t}{r}\Big)\frac{ \|f\|_{p,w;\B_t^+(x_0)}}{ w(\B_t^+(x_0))^{\frac{1}{p}}} \, \frac{dt}{t}.
\end{equation}
Finally,
\begin{align*}
\|{\wtl T}_{a} f\|_{p,w;\B_r^+(x_0)} & \  \leq  C [w]_{A_p}^{\frac{1}{p}} \|a\|_*\Big(\|f\|_{p,w;2\B_r^+(x_0)}
\\
& \  + w(\B_r^+(x_0))^{\frac{1}{p}} \, \int_{2r}^{\i} \Big(1+\ln \frac{t}{r}\Big)
 \frac{\|f\|_{p,w;\B_t^+(x_0)}}{ w(\B_t^+(x_0))^{\frac{1}{p}}} \, \frac{dt}{t}\Big)\,,
\end{align*}
and the statement  follows by \eqref{sal01}.
\end{proof}

\begin{theorem} \label{theor3.3F}
Let $w \in A_{p},$ $p\in(1,\infty),$ $a \in BMO$ and $(\varphi_1,\varphi_2)$ be  such that
\begin{equation}\label{condition2}
\int_{r}^{\infty} \Big(1+\ln \frac{t}{r}\Big)\, \frac{\es_{t<s<\infty} \varphi_1(x,s)
w(\B_s(x))^{\frac{1}{p}}}{w(\B_t(x))^{\frac{1}{p}}} \, \frac{dt}{t} \le
 C \, \varphi_2(x,r)\,.
\end{equation}
Suppose ${\wtl T}_{a}$ is  a sub-linear operator satisfying  \eqref{sublcomm} and bounded on $L_{p,w}(\R^n_+)$.
Then  ${\wtl T}_{a}$ is bounded from $M_{p,\varphi_1}(\R^n_+,w)$ to $M_{p,\varphi_2}(\R^n_+,w)$ and
\begin{equation}\label{normTaf}
\|{\wtl T}_{a} f\|_{p,\varphi_2,w;\R^n_+} \leq C[w]_{A_p}^{\frac1p} \|a\|_{*} \, \|f\|_{p,\varphi_1,w;\R^n_+}
\end{equation}
with a constant independent of $f$ and $a.$
\end{theorem}
The statement of the  theorem follows by Lemma~\ref{lem5.1.} and Theorem~\ref{Hardy}
in the same manner as the proof of Theorem~\ref{3.4.}.

\section{Calder\'on-Zygmund operators in    $M_{p,\varphi}(w)$}\label{sec5}
\setcounter{theorem}{0}
\setcounter{definition}{0}
\setcounter{lemma}{0}
\setcounter{corollary}{0}

In the present section we deal with Calder\'on-Zygmund type integrals and their commutators with $BMO$ functions.
 We start with the definition of the corresponding kernel.
\begin{definition}\label{CZK}
A measurable function $\K(x,\xi):\Rn\times\Rn\sm\{0\}\to \R$ is called a variable Calder\'on-Zygmund kernel if:
\begin{itemize}
\item[$i)$] $\K(x,\cdot)$ is a  Calder\'on-Zygmund kernel for almost all $x\in\Rn:$
\begin{itemize}
\item[$i_a)$] $\K(x,\cdot)\in C^\i(\Rn\sm\{0\}),$
\item[$i_b)$] $\K(x,\mu\xi)=\mu^{-n}\K(x,\xi)$\quad $\forall \mu>0,$
 \item[$i_c)$] $\ds \int_{\SS^{n-1}}\K(x,\xi)d\sigma_\xi=0$\quad  $\ds \int_{\SS^{n-1}}|\K(x,\xi)|d\sigma_\xi<+\i,$
\end{itemize}
\item[$ii)$] $\ds \max_{|\beta|\leq 2n}\big\|D^\beta_\xi \K\big\|_{\infty; \Rn\times\SS^{n-1}}=M<\i.$
\end{itemize}
\end{definition}
The  singular integrals
\begin{align*}
\KK f(x):=& P.V.\int_{\Rn}\K(x,x-y)f(y)\,dy\\
\CC[a, f](x):=& P.V.\int_{\Rn}\K(x,x-y)[a(x)-a(y)]f(y)\,dy\\
=& a\KK f(x)-\KK(af)(x)
\end{align*}
are bounded in $L_{p,w}$ (see  \cite{KarGulSer} for more references) and satisfy  \eqref{subl1} and \eqref{sublcomm}.
Hence the next results hold as a simple application of the estimates from Sections~\ref{sec3a} and~\ref{sec3} 
  (see \cite{GulSoft1} for details).
\begin{theorem}\label{CZcont}
Let  $w \in A_{p},$  $p\in(1,\infty)$ and $\varphi$ be weight   such that for all $x\in\Rn$ and $r>0$
\begin{equation}\label{weight}
\int_{r}^{\infty} \Big(1+\ln \frac{t}{r}\Big)\, \frac{\es_{t<s<\infty} \varphi(x,s)
 w(\B_s(x))^{\frac{1}{p}}}{w(\B_t(x))^{\frac{1}{p}}} \, \frac{dt}{t} \le C \,\varphi(x,r).
\end{equation}
Then for any $f\in M_{p,\varphi}(\Rn,w)$ and $a\in BMO$   there exist  constants  depending on $n,p,\varphi, w,$ and the kernel such that
\begin{align} \label{sal22}
\nonumber
\|\KK f\|_{p,\varphi,w}
 & \leq C [w]_{A_p}^{\frac{1}{p}} \|f\|_{p,\varphi,w}\,,\\
  \|\CC[a,f]\|_{p,\varphi,w} &  
\leq C [w]_{A_p}^{\frac{1}{p}} \|a\|_\ast\|f\|_{p,\varphi,w}\,.
\end{align}
\end{theorem}
The assertion follows by \eqref{normTf} and \eqref{normTaf}.
\begin{corollary}
Let $\Omega\subset\Rn,$  $\partial\Omega\in C^{1,1},$
$\K:\Omega\times \Rn\sm\{0\}\to \R$ be as in Definition~\ref{CZK}, $a\in BMO(\Omega)$ and $f\in M_{p,\varphi}(\Omega,w)$ with $p$, $\varphi,$ and $w$  as in Theorem~\ref{CZcont}. Then
\begin{align}\label{normO}
\nonumber
\|\KK f\|_{p,\varphi,w;\Omega} & 
\leq C [w]_{A_p}^{\frac{1}{p}} \|f\|_{p,\varphi,w;\Omega}\,,\\
 \|\CC[a,f]\|_{p,\varphi,w;\Omega} & 
\leq C [w]_{A_p}^{\frac{1}{p}} \|a\|_\ast\|f\|_{p,\varphi,w;\Omega}
\end{align}
with $C=C(n,p,\varphi,[w]_{A_p},|\Omega|, \K).$
\end{corollary}
\begin{corollary} \label{VBN}(see  \cite{ChFraL1, GulSoft1})
Let $p$,  $\varphi,$ and $w$ be as in Theorem~\ref{CZcont} and $a\in VMO$ with a $VMO$-modulus $\gamma_a.$
Then for any $\varepsilon>0$ there exists a positive number
$\rho_0=\rho_0(\varepsilon,\gamma_a)$ such that for any ball $\B_r$
with a radius $r\in(0,\rho_0)$
and all $f\in M_{p,\varphi}(\B_r,w)$
\begin{equation}\label{normB}
\|\CC[a,f]\|_{p,\varphi,w;\B_r}
\leq  C\varepsilon\|f\|_{p,\varphi,w;\B_r},
\end{equation}
with  $C$ independent of $\varepsilon$, $f,$ and $r.$
\end{corollary}

For any $x,y\in \R^n_+$    define the {\it generalized reflection\/}	${\T}(x;y)$
\begin{equation}\label{GR}
{\T}(x;y) =x-2x_n\frac{{\bf a}^n(y)}{a^{nn}(y)}\qquad
{\T}(x)={\T}(x;x):{\R}^n_+\to {\R}^n_-
\end{equation}
where ${\bf a}^n$ is  the last row of the  matrix ${\bf a}=\{a^{ij}\}_{i,j=1}^n.$
 Then
there exist positive  constants $C_1, C_2$  dependent on $n$	and
$\Lambda,$ such that
\begin{equation}\label{CTC}
C_1|\wtl x - y| \leq |{\T}(x)-y| \leq C_2 |\wtl x -y|\qquad \forall \  x,y\in{\R}_+^{n}.
\end{equation}
Then  the nonsingular integrals
\begin{align}\label{KCf}
\wtl   {\KK}	f(x)& :=\int_{{\R}^{n}_+} \K (x,{\T}(x)-y)f(y)\, dy\\
\nonumber
\wtl   {\CC}  [a,f](x)& :=\int_{{\R}^{n}_+}
\K(x,{\T}(x)-y)[a(x)-a(y)]f(y)\, dy
\end{align}
are sub-linear and according to the results in Sections~\ref{sec3} and~\ref{sec4} we have.
\begin{theorem}\label{nonsing}
Let  $a \in BMO(\R^n_+),$ $w \in A_{p}$, $p\in(1,\infty)$ and $\varphi$ be Morrey weight satisfying  \eqref{weight}.
Then   $\wtl\KK f$ and $\wtl\CC[a, f]$ are continuous in $M_{p,\varphi}(\R^n_+,w)$  and for all $f\in M_{p,\varphi}(\R^n_+,w) $ holds
\begin{equation}\label{KC}
\|\wtl\KK f\|_{p,\varphi,w;\R^n_+} \leq   C [w]_{A_p}^{\frac{1}{p}}  \|f\|_{p,\varphi,w;\R^n_+}\quad  \|\wtl\CC[a, f]\|_{p,\varphi,w;\R^n_+} \leq
 C [w]_{A_p}^{\frac{1}{p}} \|a\|_\ast \, \|f\|_{p,\varphi,w;\R^n_+}
\end{equation}
with  constants dependent on known quantities only.
\end{theorem}
\begin{corollary} \label{localnonsing}(see  \cite{ChFraL1, GulSoft1})
Let $p$,  $\varphi$ and $w$  be as in Theorem~\ref{nonsing} and $a\in VMO$ with a $VMO$-modulus $\gamma_a.$
Then for any $\varepsilon>0$ there exists a positive number
$\rho_0=\rho_0(\varepsilon,\gamma_a)$ such that for any ball $\B_r^+$
with a radius $r\in(0,\rho_0)$
and all $f\in M_{p,\varphi}(\B_r^+,w)$
\begin{equation}\label{tlKl}
\|\CC[a,f]\|_{p,\varphi,w;\B^+_r}
\leq  C\varepsilon\|f\|_{p,\varphi,w;\B^+_r},
\end{equation}
where $C$ is independent of
$\varepsilon$, $f$ and $r$.
\end{corollary}

\section{The Dirichlet problem}\label{sec6}
\setcounter{theorem}{0}
\setcounter{definition}{0}
\setcounter{lemma}{0}
\setcounter{corollary}{0}

Let $\Omega\subset\R^n,$ $n\geq 3$ be a bounded $C^{1,1}$-domain. We consider the problem
\begin{equation} \label{DP}
\begin{cases}
Lu=    a^{ij}(x)D_{ij} u +b^i(x)D_iu+c(x)u=f(x) \quad \text{ a.a.  } x\in  \Omega,\\
   u\in\       W^2_{p,\varphi}(\Omega,w) \cap \overset{\circ}W{}^1_p(\Omega,w), \   p\in(1,\infty)
  \end{cases}
\end{equation}
 subject to the following conditions:
\begin{itemize}
\item[$H_1)$]  {\it Strong ellipticity:}  there exists a constant $\Lambda>0,$ such that
\begin{equation} \label{sal12}
\begin{cases}
\ds\Lambda^{-1}|\xi|^2\leq a^{ij}(x)\xi_i\xi_j\leq\Lambda|\xi|^2 &  \text{ a.a. } x\in\Omega,\  \forall\, \xi\in \R^n\\
   a^{ij}(x)=a^{ji}(x) & 1\leq i,j\leq n.
\end{cases}
\end{equation}
Let
 $\ba=\{a^{ij}\},$ then   $\ba\in L_{\infty}(\Omega)$ and  $\|\ba\|_{\i,\Omega}=\sum_{ij=1}^n\, \|a^{ij}\|_{\infty;\Omega}$
by \eqref{sal12}.

\item [$H_2)$]{\it Regularity of the data:}  $\ba\in VMO(\Omega)$ with $VMO$-modulus   $\gamma_{\ba}:=\sum\gamma_{a^{ij}}, $ $b^i, c\in L_\infty(\Omega),$  and $f\in M_{p,\varphi}(\Omega,w)$ with $w\in A_p,$
$1<p<\i$  and  $\varphi:\Omega\times\R_+\to \R_+$ measurable.
\end{itemize}

Let  $\L=a^{ij}(x)D_{ij},$   then $\L u= f(x)-b^i(x)D_i u(x)- c(x)u.$   As it is well known  (see \cite{ChFraL1, GulSoft1} and the references therein)  for any $x\in \supp\, u, $ a ball $\B_r\subset \Omega'$ and  a function  $v\in C_0^\infty(\B_r)$ we have the representation
 \begin{align}\label{IRF}
\nonumber
D_{ij}v(x)= &\  P.V.\int_{{\B}_r}\Gamma_{ij}(x,x-y)
\left[{\L} v(y)+\big(a^{hk}(x)-a^{hk}(y)\big)
D_{hk}v(y) \right]dy\\
&+{\L} v(x)\int_{{\SS}^{n-1}}\Gamma_j(x,y) y_i d\sigma_y\\
\nonumber
= &\   {\KK}_{ij}{\L} v(x)+{\CC}_{ij}[a^{hk},D_{hk}v](x) +
{\L} v(x)\int_{{\SS}^{n-1}}\Gamma_j(x;y)y_id\sigma_y
\end{align}
According to Remark~\ref{rem1} the  formula \eqref{IRF} holds true also for functions $v\in W^2_{p,w}(\B_r).$ Here
 $\Gamma_{ij}(x,\xi)=\partial^2 \Gamma(x,\xi)/\partial\xi_i\partial\xi_j$ and $\Gamma_{ij}$ are variable Calder\'on-Zygmund kernels as in Definition~\ref{CZK}  for all $1\leq i,j\leq n.$   Then the operators $\KK_{ij}$ and $\CC_{ij}$ are singular as $\KK$ and $\CC .$
In view of  the results obtained in Section~\ref{sec5} we get for $r $ small enough
$$
\|D^2v\|_{p,\varphi,w;\B_r}	\leq C \left(\varepsilon
\|D^2v\|_{p,\varphi,w;\B_r} +\|{\L} v\|_{p,\varphi,w;\B_r} \right)\,.
$$
Choosing  $r$ such that $C\varepsilon<1$ we can  move the norm of $D^2v$ on the left-hand side and   write
\begin{equation}\label{D2est}
\|D^2v\|_{p,\varphi,w;\B_r}	\leq C \|{\L} v\|_{p,\varphi,w;\B_r} \,.
\end{equation}
Take a    cut-off function  $\eta(x)\in C_0^\infty({\B}_r)$
$$
\eta(x)=\begin{cases}
1 &  x\in {\B}_{\theta r}\\
0 &  x\not\in {\B}_{\theta'r}
\end{cases}
$$
such that   $\theta'=\theta(3-\theta)/2>\theta$   for  $ \theta\in(0,1)$  and $ |D^s\eta|\leq C [\theta(1-\theta)r]^{-s}$ for $s=0,1,2.$
Apply  \eqref{D2est} to $v(x)=\eta(x) u(x)\in W^2_{p,w}(\B_r)$  we get
\begin{align*}
\|D^2u\|_{p,\varphi,w; \B_{\theta r}} & \leq  \| D^2 v\|_{p,\varphi,w;\B_{\theta'r}}\leq C\|\L v \|_{p,\varphi,w;\B_{\theta' r}}\\[10pt]
& \leq  C\left(\|{\L} u\|_{p,\varphi,w;\B_{\theta' r}}+\frac{\|Du\|_{p,\varphi,w;\B_{\theta' r}}}{\theta(1-\theta)r}
+\frac{\|u\|_{p,\varphi,w;\B_{\theta' r}}}{[\theta(1-\theta)r]^2}  \right)\,.
\end{align*}
Since  $1< \frac{1}{\theta(1-\theta)r}$  for $r<4$ and
\begin{equation}\label{Lu-est}
\|{\L} u\|_{p,\varphi,w;\B_{\theta' r}}\leq C\big( \|L u\|_{p,\varphi,w;\B_{\theta' r}}+
\|D u\|_{p,\varphi;w,\B_{\theta' r}} +\| u\|_{p,\varphi;w,\B_{\theta' r}}\big)
\end{equation}
we can write
$$
\|D^2u\|_{p,\varphi,w; \B_{\theta r}}  \leq  C\left(\|L u\|_{p,\varphi,w;\B_{\theta' r}}+\frac{\|Du\|_{p,\varphi,w;\B_{\theta' r}}}{\theta(1-\theta)r}
+\frac{\|u\|_{p,\varphi,w;\B_{\theta' r}}}{[\theta(1-\theta)r]^2}  \right)\,.
$$
Consider now the weighted semi-norms
$$
\Theta_s = \sup_{0<\theta<1} \big[\theta(1-\theta)r \big]^s \|D^s u \|_{p,\varphi,w;\B_{\theta r}}\qquad s=0,1,2.
$$
Because of the choice of $\theta'$ we have  $\theta(1-\theta)\leq 2\theta'(1-\theta').$
Thus, after standard transformations and  taking the supremum with respect to $\theta\in(0,1)$
we get
\begin{equation}\label{theta}
\Theta_2 \leq C \left(r^2\|L u\|_{p,\varphi,w;\B_{\theta'r}} +\Theta_1+\Theta_0 \right)\,.
\end{equation}

\begin{lemma}[Interpolation inequality]\label{interpolation}
There exists a constant  $C$ independent of $r$ such that
$$
\Theta_1\leq \varepsilon \Theta_2+\frac{C}{\varepsilon}\Theta_0\qquad \text{ for any } \varepsilon\in(0,2).
$$
\end{lemma}
\begin{proof}
For functions  $u\in W^2_{p,w}(\B_r),$ $p\in(1,\infty)$ and $w\in A_p$ we dispose with the following interpolation inequality proved in \cite{Ka}
$$
\|Du\|_{p,w;\B_r}\leq C\left(\|u\|_{p,w;\B_r}+\| u\|^{\frac12}_{p,w;\B_r}\|D^2u\|^{\frac12}_{p,w;\B_r}  \right)\,.
$$
Then for any $\epsilon>0$ we have
$$
\|Du\|_{p,w;\B_r}\leq C\left(\Big( 1+\frac1{2\epsilon} \Big)\|u\|_{p,w;\B_r}+\frac{\epsilon}{2}\|D^2u\|_{p,w;\B_r}  \right)\,.
$$
Choosing $\epsilon$ small enough, such that  $\delta=\frac{C\epsilon}{2}<1,$   dividing all terms of $\varphi(x,r)w(\B_r)^{\frac1p}$ and taking the supremum over $\B_r$ we get the desired interpolation inequality in $M_{p,\varphi}(w)$
\begin{equation}\label{interp}
\|D u\|_{p,\varphi,w; \B_r}\leq \delta\|D^2 u\|_{p,\varphi,w;\B_r} +\frac{C}{\delta}\|u\|_{p,\varphi,w;\B_r}\,.
\end{equation}
We can always find some  $\theta_0\in(0,1)$ such that
\begin{align*}
\Theta_1\leq &\  2[\theta_0(1-\theta_0)r] \|Du\|_{p,\varphi,w;\B_{\theta_0r}}
\\
\leq &\  2[\theta_0(1-\theta_0)r]\left( \delta \|D^2 u\|_{p,\varphi,w;\B_{\theta_0r}} +
\frac{C}{\delta}\|u\|_{p,\varphi,w;\B_{\theta_0r}}  \right)\,.
\end{align*}
The assertion follows choosing $\delta =\frac\varepsilon2[\theta_0(1-\theta_0)r]<\theta_0r$ for any $\varepsilon\in(0,2).$
\end{proof}
Interpolating $\Theta_1$ in \eqref{theta} and taking $\theta=\frac12$ as in \cite{GulSoft1} we get the Caccioppoli-type estimate
$$
\|D^2u\|_{p,\varphi,w;\B_{r/2}}\leq C\big( \|L u\|_{p,\varphi,w; \B_r}+\frac1{r^2}\|u\|_{p,\varphi,w;\B_r}  \big)\,.
$$
Further, proceeding as in \cite{GulSoft1} and making use of  \eqref{Lu-est} and  \eqref{interp}
 we get the following interior a priori estimate.  
\begin{theorem}[Interior estimate]  \label{th-interior}
Let $u\in W^{2,{\rm loc}}_{p,w}(\Omega)$ and $L$  be a linear  elliptic operator verifying $H_1)$ and $H_2)$ such that  $L u\in M^{\rm loc}_{p,\varphi}(\Omega,w) $ with $p\in (1,\infty)$,
 $w \in A_{p}$ and $\varphi$ satisfying \eqref{weight}.  Then  $D_{ij}u\in L_{p,\varphi}(\Omega',w)$ for any $\Omega'\subset\subset\Omega''\subset\subset\Omega$  and
\begin{equation}\label{upl}
\|D^2 u\|_{p,\varphi,w;\Omega'}\leq C\big( \|u\|_{p,\varphi,w;\Omega''} + \| L u\|_{p,\varphi,w;\Omega''}\big)
\end{equation}
 where the  constant depends  on known quantities and ${\rm dist}\, (\Omega',\partial\Omega'').$
\end{theorem}

Let $x^0=(x',0)$ and denote by $C^\gamma$ the space of functions $u\in  C_0^\infty(\B_r(x^0))$ with $u=0$ for $x_n\leq 0.$ The space $W^{2,\gamma}_{p,w}(\B_r(x^0))$ is the closure of $C^\gamma$ with respect to the norm of $W^2_{p,w}. $   Then
for any  $v\in W_{p,w}^{2,\gamma}(\B_r^+(x^0))$  the next   representation formula  holds (see \cite{ChFraL2})
\begin{align}\label{bdrep}
\nonumber
D_{ij}v(x)=& \KK_{ij}\L v(x) +\CC_{ij}[a^{hk}D_{hk}v](x)\\
\nonumber
+&  {\L} v(x)\int_{{\SS}^{n-1}}\Gamma_j(x,y)y_i d\sigma_y +I_{ij}(x)\quad \forall\ i,j=1,\ldots,n,
\end{align}
where we have set
\begin{align*}
I_{ij}(x)=&\  \wtl\KK_{ij} \L v(x)  +\wtl\CC_{ij}[a^{hk},D_{hk} v](x),\qquad    \forall  \   i,j=1,\ldots,n-1,\\
I_{in}(x)=&\   I_{ni}(x)= \wtl\KK_{il}(D_n{\T}(x))^l \L v(x) +\CC_{il}[a^{hk},D_{hk}v](x) (D_n{\T}(x))^l  \\
&\qquad\qquad\qquad \forall\  i=1,\ldots, n-1,\\
 I_{nn}(x)=&\     \wtl\KK_{ls}(D_n{\T}(x))^l (D_n{\T}(x))^s \L v(x)\\
 &\qquad    + 
 \wtl\CC_{ls}[a^{hk},D_{hk}v(x)](D_n{\T}(x))^l (D_n{\T}(x))^s
\end{align*}
where
$$
D_n{\T}(x)=\left( (D_n{\T}(x))^1,\ldots,(D_n{\T}(x))^n \right) ={\T}(e_n,x).
$$
Applying the  estimates \eqref{KC}  and \eqref{tlKl},     the interpolation inequality   \eqref{interp} and  taking into account the $VMO$ properties of the coefficients $a^{ij}$'s,
it is possible to choose  $r_0$ small enough such that 
\begin{equation}\label{bdrest2}
\|D_{ij}v\|_{p,\varphi;w,{\B}_r^+}\leq C( \|L v\|_{p,\varphi; w,{\B}_r^+}+  \|u\|_{p,\varphi; w,{\B}_r^+})
\end{equation}
for all $r<r_0$ (see \cite{GulSoft1} for details).
 By local flattering of the boundary, covering with semi-balls, taking a   partition of  unity  subordinated  to that covering and applying the   estimate  \eqref{bdrest2} we get a boundary a priori estimate that unified with
\eqref{upl}  gives the next  theorem.
\begin{theorem}[Main result]\label{mine}
Let $u\in  W^2_{p,\varphi}(\Omega,w)\cap \overset{\circ}{W}{}^1_p(\Omega,w)   $     be a solution of \eqref{DP} under the  conditions $H_1)$ and $H_2).$ Then  the next estimate holds   for any  $w\in A_p,$ $p\in(1,\infty)$ and $\varphi$ satisfying \eqref{weight}
\begin{equation}\label{est}
\|D^2u\|_{p,\varphi,w;\Omega}\leq C\big(\|u\|_{p,\varphi,w;\Omega} +  \| f\|_{p,\varphi,w; \Omega}\big)
\end{equation}
and the constant $C$ depends on known quantities only.
\end{theorem}

Let us note that  the solution  of \eqref{DP} exists according to Remark~\ref{rem1}. The a priori estimate follows  as in
 \cite{ChFraL1, ChFraL2} making use of \eqref{Lu-est}  and the interpolation inequality in weighted Lebesgue spaces \cite{Ka}.

\subsection*{Acknowledgments}
The research of V. Guliyev and M. Omarova is  partially supported by the grant of
Science Development Foundation under the President of the Republic of Azerbaijan, project EIF-2013-9(15)-FT.
The research of V. Guliyev is  partially supported by the grant of Ahi Evran University Scientific Research Projects (PYO.FEN.4003-2.13.007).

L. Softova is a member of GNAMPA-INDAM. The present work has  prepared during the visit of the third author at the Ahi Evran University for which she expresses her gratitude at the staff of the Department of Mathematics for the kind hospitality.

\end{document}